\documentclass[11pt]{amsart}
\usepackage{amssymb}
\usepackage{verbatim}
\usepackage{hyperref}
\usepackage{color}
\begin{document}
\newtheorem{theorem}{Theorem}    
\newtheorem{proposition}[theorem]{Proposition}
\newtheorem{conjecture}[theorem]{Conjecture}
\def\theconjecture{\unskip}
\newtheorem{corollary}[theorem]{Corollary}
\newtheorem{lemma}[theorem]{Lemma}
\newtheorem{sublemma}[theorem]{Sublemma}
\newtheorem{observation}[theorem]{Observation}
\theoremstyle{definition}
\newtheorem{definition}{Definition}
\newtheorem{notation}[definition]{Notation}
\newtheorem{remark}[definition]{Remark}
\newtheorem{question}[definition]{Question}
\newtheorem{questions}[definition]{Questions}
\newtheorem{example}[definition]{Example}
\newtheorem{problem}[definition]{Problem}
\newtheorem{exercise}[definition]{Exercise}

\numberwithin{theorem}{section} \numberwithin{definition}{section}
\numberwithin{equation}{section}

\def\earrow{{\mathbf e}}
\def\rarrow{{\mathbf r}}
\def\uarrow{{\mathbf u}}
\def\varrow{{\mathbf V}}
\def\tpar{T_{\rm par}}
\def\apar{A_{\rm par}}

\def\reals{{\mathbb R}}
\def\torus{{\mathbb T}}
\def\heis{{\mathbb H}}
\def\integers{{\mathbb Z}}
\def\naturals{{\mathbb N}}
\def\complex{{\mathbb C}\/}
\def\distance{\operatorname{distance}\,}
\def\support{\operatorname{support}\,}
\def\dist{\operatorname{dist}\,}
\def\Span{\operatorname{span}\,}
\def\degree{\operatorname{degree}\,}
\def\kernel{\operatorname{kernel}\,}
\def\dim{\operatorname{dim}\,}
\def\codim{\operatorname{codim}}
\def\trace{\operatorname{trace\,}}
\def\Span{\operatorname{span}\,}
\def\dimension{\operatorname{dimension}\,}
\def\codimension{\operatorname{codimension}\,}
\def\nullspace{\scriptk}
\def\kernel{\operatorname{Ker}}
\def\ZZ{ {\mathbb Z} }
\def\p{\partial}
\def\rp{{ ^{-1} }}
\def\Re{\operatorname{Re\,} }
\def\Im{\operatorname{Im\,} }
\def\ov{\overline}
\def\eps{\varepsilon}
\def\lt{L^2}
\def\diver{\operatorname{div}}
\def\curl{\operatorname{curl}}
\def\etta{\eta}
\newcommand{\norm}[1]{ \|  #1 \|}
\def\expect{\mathbb E}
\def\bull{$\bullet$\ }
\def\C{\mathbb{C}}
\def\R{\mathbb{R}}
\def\Rn{{\mathbb{R}^n}}
\def\Sn{{{S}^{n-1}}}
\def\M{\mathbb{M}}
\def\N{\mathbb{N}}
\def\Q{{\mathbb{Q}}}
\def\Z{\mathbb{Z}}
\def\F{\mathcal{F}}
\def\L{\mathcal{L}}
\def\S{\mathcal{S}}
\def\essinf{\operatorname{essinf}}
\def\esssup{\operatorname{esssup}}
\def\esslimsup{\operatorname{esslimsup}}
\def\essliminf{\operatorname{essliminf}}
\def\supp{\operatorname{supp}}
\def\dist{\operatorname{dist}}
\def\pv{\operatorname{p.v.}}
\def\essi{\operatornamewithlimits{ess\,inf}}
\def\esss{\operatornamewithlimits{ess\,sup}}
\def\xone{x_1}
\def\xtwo{x_2}
\def\xq{x_2+x_1^2}
\newcommand{\abr}[1]{ \langle  #1 \rangle}

\newcommand{\Norm}[1]{ \left\|  #1 \right\| }
\newcommand{\set}[1]{ \left\{ #1 \right\} }
\def\one{\mathbf 1}
\def\whole{\mathbf V}
\newcommand{\modulo}[2]{[#1]_{#2}}

\def\scriptf{{\mathcal F}}
\def\scriptg{{\mathcal G}}
\def\scriptm{{\mathcal M}}
\def\scriptb{{\mathcal B}}
\def\scriptc{{\mathcal C}}
\def\scriptt{{\mathcal T}}
\def\scripti{{\mathcal I}}
\def\scripte{{\mathcal E}}
\def\scriptv{{\mathcal V}}
\def\scriptw{{\mathcal W}}
\def\scriptu{{\mathcal U}}
\def\scriptS{{\mathcal S}}
\def\scripta{{\mathcal A}}
\def\scriptr{{\mathcal R}}
\def\scripto{{\mathcal O}}
\def\scripth{{\mathcal H}}
\def\scriptd{{\mathcal D}}
\def\scriptl{{\mathcal L}}
\def\scriptn{{\mathcal N}}
\def\scriptp{{\mathcal P}}
\def\scriptk{{\mathcal K}}
\def\frakv{{\mathfrak V}}

\title[A Tracing of the Fractional Temperature Field]
{A Tracing of the Fractional Temperature Field}

\author[S.G. Shi]{S.G. Shi } 

\author[J. Xiao]{J. Xiao}

\subjclass[2010]{
Primary 31C15, 42B20; Secondary 32A35; 35K08.
}
%
\keywords{Fractional temperature field; capacity; maximal function}
\thanks{S.G. Shi was supported by NSF of China (Grant Nos. 11301249, 11271175) and the Applied Mathematics Enhancement Program of Linyi University(No. LYDX2013BS059); J. Xiao was supported by NSERC of Canada (FOAPAL: 202979463102000) and URP of Memorial University
(FOAPAL: 208227463102000).}

\address{S.G. Shi\\Department of Mathematics\\
Linyi University \\
Linyi 276005\\
P. R. China}
\email{shishaoguang@mail.bnu.edu.cn}

\address{J. Xiao\\Department of Mathematics and Statistics\\
Memorial University \\
St. John¡¯s, NL A1C 5S7\\
Canada}
\email{jxiao@mun.ca}


\maketitle

\begin{abstract}
This note is devoted to a study of $L^q$-tracing of the fractional temperature field $u(t,x)$ -- the weak solution of the fractional heat equation $(\partial_t+(-\Delta_x)^\alpha)u(t,x)=g(t,x)$ in $L^p(\mathbb R^{1+n}_+)$ subject to the initial temperature $u(0,x)=f(x)$ in $L^p(\mathbb R^n)$.
\end{abstract}

\section{Introduction}\label{section1} 

Directly continuing from \cite{CX, JXYZ}, we consider the fractional heat equation in the upper-half Euclidean space $\mathbb R_+^{1+n}=\mathbb R_+\times\mathbb R^n$ with $\mathbb R_+=(0,\infty)$ and $n\ge 1$:
\begin{equation}\label{1.1}
\begin{cases}
\big(\partial_t+(-\Delta_x)^{\alpha}\big)u(t,x)=g(t,x)\,\,\,\,\forall\,\,(t,x)\in\mathbb{R}_{+}^{1+n};\\
u(0,x)=f(x)\,\,\,\,\forall\,\,x\in\mathbb{R}^{n},
\end{cases}
\end{equation}
where $(-\Delta_x)^{\alpha}$ denotes the fractional ($0<\alpha<1$) power of the spatial Laplacian that is determined by
$$
(-\Delta_x)^{\alpha}u(\cdot,x)=\mathcal{F}^{-1}(|\xi|^{2\alpha}\mathcal{F}u(\cdot,\xi))(x)\,\,\,\,\forall\,\, x\in \mathbb{R}^{n}
$$
for which $\mathcal{F}$ is the Fourier transform and $\mathcal{F}^{-1}$ is its inverse. Specifically, we are interested in the trace of such a fractional temperature field (existing as the weak solution of \eqref{1.1})
$$
u(t,x)=R_\alpha f(t,x)+S_\alpha g(t,x)
$$
with
$$
\begin{cases}
R_{\alpha}f(t,x)=e^{-t(-\Delta_x)^{\alpha}}f(x)=\int_{\mathbb{R}^{n}}K_{t}^{(\alpha)}(x-y)f(y)dy;\\
S_{\alpha}g(t,x)=\int_{0}^{t}e^{-(t-s)(-\Delta_x)^{\alpha}}g(s,x)ds=\int_{\mathbb{R}^n}\big(\int_0^t K_{t-s}^{(\alpha)}(x-y)g(s,y)ds\big)dy,
\end{cases}
$$
where $K_{t}^{(\alpha)}(x)$ is the fractional heat kernel
$$
K_{t}^{(\alpha)}(x)\equiv(2\pi)^{-\frac{n}{2}}\int_{\mathbb{R}^{n}}e^{ix\cdot y-t|y|^{2\alpha}}dy \,\,\,\,\,\forall\,\,(t,x)\in\mathbb{R}_{+}^{1+n}
$$
whose endpoint $\alpha=1$ and middle-point $\alpha=1/2$ lead to the heat kernel and Poisson kernel:
$$
K_{t}^{(1)}(x)=(4\pi)^{-\frac{n}{2}}e^{-\frac{|x|^{2}}{4t}}\,\,\,\,\&\,\,\,\,K_{t}^{(\frac{1}{2})}(x)=\pi^{-\frac{n+1}{2}}\Gamma \Big(\frac{n+1}{2}\Big)\frac{t}{(t^{2}+|x|^{2})^{\frac{n+1}{2}}}
$$
with $\Gamma(\cdot)$ being the classical gamma function. Although there is no explicit formula for $K_{t}^{(\alpha)}(x)$ under $\alpha\in (0,1)\setminus\{1/2\}$ (cf. \cite{CDDF, CW, MYZ, NSS, NSY, NY, W1, W2, WY, Z2}), the following estimates are not only valid but also practical (cf. \cite{ARAG, BC, BG, CS, XZ}):
$$
\begin{cases}
K_{t}^{(\alpha)}(x)\approx \min\{t^{-\frac{n}{2\alpha}},t|x|^{-(n+2\alpha)}\}\approx \frac{t}{(t^{\frac{1}{2\alpha}}+|x|)^{n+2\alpha}}\,\,\,\,\forall\,\,(t,x)\in\mathbb{R}_{+}^{1+n};\\
\int_{\mathbb{R}^{n}}K_{t}^{(\alpha)}(x)dx=1\,\,\,\,\forall\,\,t\in (0,\infty).
\end{cases}
$$

As explored in \cite{CX, JXYZ}, the regularity of $u(t,x)$ sheds some light on the traces/restrictions of $R_\alpha f(t,x)$ and $S_\alpha g(t,x)$ to subsets of $\mathbb R^{1+n}_+$ of $(1+n)$-dimensional Lebesgue measure zero. Here $f(x)$ and $g(t,x)$ are arbitrary functions of the usual Lebesgue classes $L^p(\mathbb R^n)$ and $L^{p}(\mathbb R^{1+n}_+)$, respectively. In order to characterize the traces of $R_\alpha f(t,x)$ and $S_\alpha g(t,x)$ on a given compact exceptional set $K\subset\mathbb R^{1+n}_+$, we investigate nonnegative Radon measures supported on $K$ such that under $1<p,q<\infty$ the mapping $R_\alpha: L^p(\mathbb R^n)\mapsto L^q_\mu(\mathbb R^{1+n}_+)$ and $S_\alpha: L^p(\mathbb R^{1+n}_+)\mapsto L^q_\mu(\mathbb R^{1+n}_+)$ are continuous - namely -
\begin{equation}
\label{eR}
\left(\int_{\mathbb R^{1+n}_+}|R_\alpha f(t,x)|^q\,d\mu(t,x)\right)^\frac1q\lesssim \|f\|_{L^p(\mathbb R^n)}
\end{equation}
and
\begin{equation}
\label{eS}
\left(\int_{\mathbb R^{1+n}_+}|S_\alpha g(t,x)|^q\,d\mu(t,x)\right)^\frac1q\lesssim \|g\|_{L^p(\mathbb R^{1+n}_+)},
\end{equation}
where the symbol $A\lesssim B$ means $A\leq cB$ for a positive constant $c$ - moreover - $A\approx B$ stands for both $A\lesssim B$ and $B\lesssim A$.

A careful examination of \eqref{eR} and \eqref{eS} indicates that they can be naturally unified as:
\begin{equation}
\label{eT}
\left(\int_{\mathbb R^{1+n}_+}|T_\alpha h|^q\,d\mu\right)^\frac1q\lesssim \|h\|_{L^p(\mathbb X)}=\begin{cases} \|f\|_{L^p(\mathbb R^{n})}\ \hbox{as}\ (T_\alpha,h,\mathbb X)=(R_\alpha,f,\mathbb R^n);\\
\|g\|_{L^p(\mathbb R^{1+n}_+)}\ \hbox{as}\ (T_\alpha,h,\mathbb X)=(S_\alpha,g,\mathbb R^{1+n}_+).
\end{cases}
\end{equation}
Describing such a measure $\mu$ on $\mathbb R^{1+n}_+$ depends on a concept of the induced capacity. For a compact set $K\subset\mathbb R^{1+n}_+$ let
$$
C_{p}^{(T_{\alpha})}(K)=\inf\{\|h\|_{L^{p}(\mathbb X)}^{p}: h\geq 0 \,\,\&\,\,T_{\alpha} h\geq \textbf{1}_{K}\},
$$
where $\textbf{1}_{K}$ is the characteristic function of $K$. Then, for an open subset $O$ of $\mathbb{R}_{+}^{1+n}$ let
$$
C_{p}^{(T_{\alpha})}(O)=\sup\{C_{p}^{(T_{\alpha})}(K): \hbox{compact} \,K \subset O\},
$$
and hence for any set $E\subset \mathbb{R}_{+}^{1+n}$ let
$$
C_{p}^{(T_{\alpha})}(E)=\inf\{C_{p}^{(T_{\alpha})}(O): \hbox{open} \,O \supset E\}.
$$
According to \cite{CX, JXYZ}, if
$$
B^{(\alpha)}_{r}(t_{0}, x_{0})\equiv\{(t,x)\in \mathbb{R}_{+}^{1+n}: r^{2\alpha}<t-t_{0}<2r^{2\alpha} \,\,\&\,\, |x-x_{0}|<r\}
$$
stands for the parabolic ball with centre $(t_{0}, x_{0})\in\mathbb R^{1+n}_+$ and radius $r>0$, then
\begin{equation}
\label{capBall}
C_{p}^{(T_\alpha)}\big(B^{(\alpha)}_{r}(t_{0}, x_{0})\big)\approx\begin{cases} r^n\ \hbox{as}\ T_\alpha=R_\alpha;\\
r^{n+2\alpha(1-p)}\ \hbox{as}\ T_\alpha=S_\alpha\ \&\ 1<p<1+\frac{n}{2\alpha}.
\end{cases}
\end{equation}

Below is a tracing principle for the fractional heat equation \eqref{1.1}.

\begin{theorem}
\label{MT} Let $0<\alpha<1$ and $1<p<1+\frac{n}{2\alpha}$. Then
$$
\eqref{eT}\Leftrightarrow\begin{cases} \sup\left\{\frac{\mu\big(B^{(\alpha)}_{r}(t_{0}, x_{0})\big)}{ \big(C_p^{(T_\alpha)}\big(B^{(\alpha)}_{r}(t_{0}, x_{0})\big)\big)^{q/p}}:\ (r,t_0,x_0)\in\mathbb R_+\times\mathbb R_+\times \mathbb R^n\right\}<\infty\ \hbox{as}\ p<q;\\
\sup\left\{\frac{\mu(K)}{C_p^{(T_\alpha)}(K)}:\ \hbox{compact}\ K\subset\mathbb R^{1+n}_+\right\}<\infty\ \hbox{as}\ p=q;\\
\int_{\mathbb R^{1+n}_+}\left(\int_0^\infty \Big(\frac{\mu(B^{(\alpha)}_{r}(t_{0}, x_{0}))}{C_p^{(T_\alpha)}(B^{(\alpha)}_{r}(t_{0}, x_{0}))}\Big)^\frac1{p-1}\,\frac{dr}{r}\right)^\frac{q(p-1)}{p-q}\,d\mu(t_0,x_0)<\infty\ \hbox{as}\ p>q.
\end{cases}
$$
\end{theorem}

Here, it should be noted that $R_\alpha$-case of Theorem \ref{MT} under $p\le q$ has been treated in \cite[Theorems 3.2-3.3]{CX}. Of course, the remaining cases of Theorem \ref{MT} are new. Perhaps, it is worth to point out that under $p=q$,
$$
\sup\left\{\frac{\mu(K)}{C_p^{(T_\alpha)}(K)}:\ \hbox{compact}\ K\subset\mathbb R^{1+n}_+\right\}<\infty
$$
implies
$$
\sup\left\{\frac{\mu(B_r^\alpha(t_0,x_0))}{C_p^{(T_\alpha)}(B_r^\alpha(t_0,x_0))}:\ B_r^\alpha(t_0,x_0)\subset\mathbb R^{1+n}_+\right\}<\infty
$$
but not conversely in general - \cite[Theorem 4(ii)]{Adams} and its argument might be helpful to produce a ball-based sufficient condition for \eqref{eT} to hold. Upon $d\mu=dtdx$ in $S_\alpha$-case of Theorem \ref{MT} we have
$\mu(B_{r}^{(\alpha)}(t_{0},x_{0}))\approx r^{n+2\alpha}$, thereby finding that (cf. \cite[Theorem 1.4]{Z1}) for $g\in L^{p}(\mathbb{R}_{+}^{1+n})$ one has
$$
\|S_{\alpha}g\|_{L^{\widetilde{q}}(\mathbb{R}_{+}^{1+n})}\lesssim \|g\|_{L^{p}(\mathbb{R}_{+}^{1+n})}\,\,\,\,\hbox{where}\,\,\widetilde{q}=p\left(1+\frac{2\alpha p}{n+2\alpha-2\alpha p}\right)>p.
$$
Although $R_\alpha$ and $S_\alpha$ behave similarly, the argument for Theorem \ref{MT} will be still split into two parts - one for $R_\alpha$ in Section 2 and another one for $S_\alpha$ in Section 3 - this is because the subtle difference between  $R_\alpha$ and $S_\alpha$ can be seen clearly from such a splitting arrangement.

\section{$R_{\alpha}$'s tracing}\label{section2}

In this section we verify Theorem \ref{MT} for $T_\alpha=R_\alpha$. To do so, we need three lemmas as seen below.

The first is about the dual representation of $C_p^{(R_\alpha)}(K)$ for a given compact set $K\subset\mathbb R^{1+n}_+$.

\begin{lemma}
\label{l20} Let $\mathcal{U}^+(K)$ be the class of all nonnegative Radon measures $\mu$ with compact support $K\subset\mathbb R^{1+n}_+$ and the total variation $\|\mu\|$. Then
$$
C_{p}^{(R_\alpha)}(K)=\sup\left\{\|\mu\|^{p}:\ \ \mu\in \mathcal{U}^{+}(K) \,\,\&\,\,\int_{\mathbb R^n}\Big(\int_{\mathbb R^{1+n}_+}K_t^\alpha(x-y)\,d\mu(t,y)\Big)^\frac{p}{p-1}\,dx\leq 1\right\}.
$$
\end{lemma}
\begin{proof}
Note that
$$
\int_{\mathbb{R}_{+}^{1+n}}R_{\alpha}f(t,x)h(t,x)dtdx=\int_{\mathbb{R}^{n}}f(x)\Big(\int_{\mathbb{R}_{+}^{1+n}}K^{(\alpha)}_{t}(x-y)h(t,y)dtdy\Big)dx
$$
holds for all $(f,h)\in C_{0}^{\infty}(\mathbb{R}^{n})\times C_{0}^{\infty}(\mathbb{R}_{+}^{1+n})$ where $C^\infty_0(\mathbb X)$ stands for the class of infinitely differentiable functions with compact support in $\mathbb X=\mathbb R^n\ \hbox{or}\ \mathbb R^{1+n}_+$. Thus, the adjoint operator of $R_{\alpha}$ is defined by
$$
(R_{\alpha}^{*}h)(x)=\int_{\mathbb{R}_{+}^{1+n}}K_{t}^{(\alpha)}(x-y)h(t,y)dtdy  \,\,\,\,\forall \,\,h\in C_{0}^{\infty}(\mathbb{R}_{+}^{1+n}).
$$
For any nonnegative Radon measure $\mu$ in $\mathbb R^{1+n}_+$ and a continuous function $f$ with a compact support in $\mathbb{R}^{n}$, one has
$$
\left|\int_{\mathbb{R}_{+}^{1+n}}R_{\alpha}fd\mu\right|\lesssim \|f\|_{L^\infty(\mathbb{R}^{n})}\|\mu\|.
$$
Therefore, the Riesz representation theorem yields a Borel measure $\nu$ on $\mathbb{R}^{n}$ such that
$$
\int_{\mathbb{R}_{+}^{1+n}}R_{\alpha}fd\mu=\int_{\mathbb{R}^{n}}fd\nu \,\,\,\,\forall\,\,f\geq 0.
$$
This means that $\nu=R_{\alpha}^{*}\mu$ can be defined by
$$
R_{\alpha}^{*}\mu(x)=\int_{\mathbb{R}_{+}^{1+n}}K_{t}^{(\alpha)}(x-y)d\mu(t,y).
$$
According to \cite[Proposition 1]{CX}, one gets
$$
C_{p}^{(R_\alpha)}(K)=\sup\Big\{\|\mu\|^{p}:\ \mu\in \mathcal{U}^{+}(K) \,\,\&\,\,\|R_{\alpha}^{*} \mu\|_{L^{\frac{p}{p-1}}(\mathbb{R}^{n})}\leq 1\Big\}.
$$
\end{proof}

The second is about $L^p$-boundedness of the fractional maximal operator of parabolic type.

 \begin{lemma}\label{l21} For a nonnegative Radon measure $\mu$ on $\mathbb R^{1+n}_+$ let
 $$
  M_{\alpha}\mu(x)=\sup_{r>0}{r^{-n}}\mu\big({B^{(\alpha)}_{r}(r^{2\alpha}, x)}\big)
  $$
  be its fractional parabolic maximal function. Then
 $$
 \|M_{\alpha}\mu\|_{L^{p}(\mathbb{R}^{n})}\approx \|R^\ast_\alpha\mu\|_{L^p(\mathbb R^n)}\ \ \forall\  \ p\in (1,\infty).
 $$
 \end{lemma}
 \begin{proof} A straightforward estimation with $x\in\mathbb R^n$ and $R^\ast_\alpha\mu(x)$ gives
 $$
 R^\ast_\alpha\mu(x)\gtrsim \int_{B^{(\alpha)}_{r}(r^{2\alpha}, x)}\frac{t}{(t^{\frac{1}{2\alpha}}+|x-y|)^{n+2\alpha}}d\mu(t,y)\gtrsim \frac{\mu(B^{(\alpha)}_{r}(r^{2\alpha}, x))}{r^{n}}\ \ \forall\ \ r>0,
 $$
 whence
 $$
 R^{\alpha}_\alpha\mu(x)\gtrsim M_{\alpha}\mu(x).
 $$
 This implies
 $$
 \|M_{\alpha}\mu\|_{L^{p}(\mathbb{R}^{n})}\lesssim \|R^\ast_{\alpha}\mu\|_{L^{p}(\mathbb{R}^{n})}.
 $$

 To prove the converse inequality, we slightly modify \cite[(3.6.1)]{AH} to get two constants $a>1$ and $b>0$ such that for any $\lambda>0$ and $0<\varepsilon\leq 1$, one has the following good-$\lambda$ inequality
 \begin{align}\label{2.4}
 |\{x\in\mathbb R^n:\ R^\ast_{\alpha}\mu(x)>a\lambda\}|&\leq b\varepsilon^{\frac{n+2\alpha}{n}}|\{x\in\mathbb R^n:\ R^\ast_{\alpha}\mu(x)>\lambda\}|\nonumber\\
 &\ \ +|\{x\in\mathbb R^n:\ M_{\alpha}\mu(x)>\varepsilon\lambda\}|.
 \end{align}
 Inspired by \cite[Theorem 3.6.1]{AH}, we proceed the proof by using $(\ref{2.4})$.
 Multiplying $(\ref{2.4})$ by $\lambda^{p-1}$ and integrating in $\lambda$, we have for any $\gamma>0$,
 \begin{align*}
 \int_{0}^{\gamma}|\{x\in\mathbb R^n:\ R^\ast_{\alpha}\mu(x)>a\lambda\}|\lambda^{p-1}d\lambda&\leq b\varepsilon^{\frac{n+2\alpha}{n}}\int_{0}^{\gamma}|\{x\in\mathbb R^n:\ R^\ast_{\alpha}\mu(x)>\lambda\}|\lambda^{p-1}d\lambda\\
 &\quad+\int_{0}^{\gamma}|\{x\in\mathbb R^n:\  M_{\alpha}\mu(x)>\varepsilon\lambda\}|\lambda^{p-1}d\lambda.
 \end{align*}
 An equivalent formulation of the above inequality is
 \begin{align*}
 a^{-p}\int_{0}^{a\gamma}|\{x\in\mathbb R^n:\ R^\ast_{\alpha}\mu(x)>a\lambda\}|\lambda^{p-1}d\lambda&\leq b\varepsilon^{\frac{n+2\alpha}{n}}\int_{0}^{\gamma}|\{x\in\mathbb R^n:\ R^\ast_{\alpha}\mu(x)>\lambda\}|\lambda^{p-1}d\lambda\\
 &\,\,\,\,+\varepsilon^{-p}\int_{0}^{\varepsilon \gamma}|\{x\in\mathbb R^n:\ M_{\alpha}\mu(x)>\varepsilon\lambda\}|\lambda^{p-1}d\lambda.
 \end{align*}
 Let $\varepsilon$ be so small that $b\varepsilon^{\frac{n+2\alpha}{n}}\leq \frac{1}{2}a^{-p}$ and $\gamma\to\infty$. Then
 $$
 a^{-p}\int_{\mathbb{R}^{n}}(R^\ast_{\alpha}\mu(x))^{p}dx\leq 2\varepsilon^{-p}\int_{\mathbb{R}^{n}}(M_{\alpha}\mu(x))^{p}dx.
 $$
 That is
 $$
 \|M_{\alpha}\mu\|_{L^{p}(\mathbb{R}^{n})}\gtrsim \|R^\ast_{\alpha}\mu\|_{L^{p}(\mathbb{R}^{n})}.
 $$
 \end{proof}

 The third is about the Hedberg-Wolff potential for $R_\alpha$:
  $$
 P_{\alpha p}^{R}\mu(t,x)=\int_{0}^{\infty}\left(\frac{\mu(B^{(\alpha)}_{r}(t, x))}{r^{n}}\right)^{p'-1}\frac{dr}{r}\,\,\,\,\forall\,\,(t,x)\in \mathbb{R}_{+}^{1+n}.
 $$

\begin{lemma}\label{l22}  Let $1<p<\infty$, $p'=\frac{p}{p-1}$, and $\mu$ be a nonnegative Radon measure on $\mathbb R^{1+n}_+$. Then
$$
\|R_{\alpha}^{*}\mu\|_{L^{p'}(\mathbb{R}^{n})}^{p'}\approx \int_{\mathbb{R}_{+}^{1+n}}P_{\alpha p}^{R}\mu\, d\mu.
$$
\end{lemma}
\begin{proof} Below is a two-fold argument.

{\it Part 1}. The first task is to show
$$
\|R^\ast_{\alpha}\mu\|_{L^{p'}(\mathbb{R}^{n})}^{p'}\lesssim \int_{\mathbb{R}_{+}^{1+n}}P_{\alpha p}^{R}\mu\,d\mu.
$$
Note first that
$$
\frac{\mu(B^{(\alpha)}_{r}(r^{2\alpha}, x))}{r^{n}}\approx \left(\int_{r}^{2r}\left(\frac{\mu(B^{(\alpha)}_{s}(s^{2\alpha}, x))}{s^{n}}\right)^{p'}\frac{ds}{s}\right)^{\frac{1}{p'}}\lesssim \left(\int_{0}^{\infty}\left(\frac{\mu(B^{(\alpha)}_{s}(s^{2\alpha}, x))}{s^{n}}\right)^{p'}\frac{ds}{s}\right)^{\frac{1}{p'}}.
$$
Therefore, one has
$$
M_{\alpha}\mu(x)\lesssim \left(\int_{0}^{\infty}\left(\frac{\mu(B^{(\alpha)}_{s}(s^{2\alpha}, x))}{s^{n}}\right)^{p'}\frac{ds}{s}\right)^{\frac{1}{p'}}.
$$
By Lemma \ref{l21}, it is sufficient to verify
$$
\int_{\mathbb{R}^{n}}\int_{0}^{\infty}\left(\frac{\mu(B^{(\alpha)}_{r}(r^{2\alpha}, x))}{r^{n}}\right)^{p'}\frac{dr}{r}dx\lesssim \int_{\mathbb{R}_{+}^{1+n}}\int_{0}^{\infty}\left(\frac{\mu(B^{(\alpha)}_{r}(t, x))}{r^{n}}\right)^{p'-1}\frac{dr}{r}d\mu.
$$
Using the Fubini theorem, one has
$$
\int_{\mathbb{R}^{n}}\int_{0}^{\infty}\left(\frac{\mu(B^{(\alpha)}_{r}(r^{2\alpha}, x))}{r^{n}}\right)^{p'}\frac{dr}{r}dx=\int_{0}^{\infty}\int_{\mathbb{R}^{n}}\frac{\mu(B^{(\alpha)}_{r}(r^{2\alpha}, x))^{p'}}{r^{np'}}dx\frac{dr}{r}.
$$
A further application of Fubini's theorem yields
\begin{align*}
\int_{\mathbb{R}^{n}}\mu(B^{(\alpha)}_{r}(r^{2\alpha}, x))^{p'}dx&\lesssim\int_{B^{(\alpha)}_{r}(r^{2\alpha}, x)}\int_{\mathbb{R}^{n}}\mu(B^{(\alpha)}_{r}(r^{2\alpha}, x))^{p'-1}dxd\mu\\
 &\lesssim \int_{B^{(\alpha)}_{r}(r^{2\alpha}, x)}\int_{\mathbb{R}^{n}}\mu(B^{(\alpha)}_{r}(r^{2\alpha}, y))^{p'-1}dxd\mu\\
 &\lesssim r^{n}\int_{B^{(\alpha)}_{r}(r^{2\alpha}, x)}\mu(B^{(\alpha)}_{r}(r^{2\alpha}, y))^{p'-1}d\mu.
\end{align*}
Therefore,
\begin{align*}
\int_{0}^{\infty}\int_{\mathbb{R}^{n}}\frac{\mu(B^{(\alpha)}_{r}(r^{2\alpha}, x))^{p'}}{r^{np'}}dx\frac{dr}{r}&\approx
\int_{0}^{\infty}\int_{B^{(\alpha)}_{r}(r^{2\alpha}, x)}\frac{\mu(B^{(\alpha)}_{r}(r^{2\alpha}, y))^{p'-1}}{r^{n(p'-1)}}d\mu\frac{dr}{r}\\
&\approx \int_{B^{(\alpha)}_{r}(r^{2\alpha}, x)}\int_{0}^{\infty}\left(\frac{\mu(B^{(\alpha)}_{r}(r^{2\alpha}, y))}{r^{n}}\right)^{p'-1}\frac{dr}{r}d\mu\\
&\lesssim \int_{\mathbb{R}_{+}^{1+n}}\left(\int_{0}^{\infty}\left(\frac{\mu(B^{(\alpha)}_{r}(t, x))}{r^{n}}\right)^{p'-1}\frac{dr}{r}\right)d\mu(t,x),
\end{align*}
as desired.

{\it Part 2.} The second task is to prove
 $$
 \|R^\ast_{\alpha}\mu\|_{L^{p'}(\mathbb{R}^{n})}^{p'}\gtrsim \int_{\mathbb{R}_{+}^{1+n}}P_{\alpha p}^{R}\mu\,d\mu.
 $$

Since
\begin{align*}
\|R^\ast_{\alpha}\mu\|_{L^{p'}(\mathbb{R}^{n}, \,dx)}^{p'}&=\int_{\mathbb{R}^{n}}\big(R^\ast_{\alpha}\mu(x)\big)^{p'-1}(R^\ast_{\alpha}\mu(x))dx\\
&=\int_{\mathbb{R}_{+}^{1+n}}\int_{\mathbb{R}^{n}}(R^\ast_{\alpha}\mu(x))^{p'-1}K_{t}^{(\alpha)}(x-y)dx\,d\mu(t,y).
\end{align*}
Upon writing
$$
K(t,x)=\int_{\mathbb{R}^{n}}(R^\ast_{\alpha}\mu(x))^{p'-1}K_{t}^{(\alpha)}(x-y)dx
$$
and
$$
B(x,2^{-m})=\{y\in\mathbb R^n:\ |x-y|<2^{-m}\ \&\ (2^{-m})^{2\alpha}<t<2(2^{-m})^{2\alpha}\}\ \forall\ m\in\mathbb Z\equiv\{0,\pm 1,\pm 2,...\},
$$
we obtain
\begin{align*}
K(t,x)&\approx\int_{\mathbb{R}^{n}}\frac{t}{(t^{\frac{1}{2\alpha}}+|x-y|)^{n+2\alpha}}\left(\int_{\mathbb{R}_{+}^{1+n}}\frac{s}{(s^{\frac{1}{2\alpha}}+|y-z|)^{n+2\alpha}}d\mu\right)^{p'-1}dy\\
&\gtrsim \sum_{m\in \mathbb{Z}}\int_{B(x,2^{-m})}t^{-\frac{n}{2\alpha}}
\left(\int_{B^{(\alpha)}_{2^{-m}}(t, x)}s^{-\frac{n}{2\alpha}}d\mu\right)^{p'-1}dy\\
&\gtrsim \sum_{m\in \mathbb{Z}}\int_{B(x,2^{-m})}{2^{mn}}
\left(\frac{\mu(B^{(\alpha)}_{2^{-m}}(t, x))}{2^{-m}}\right)^{p'-1}dy\\
&\gtrsim \int_{0}^{\infty}\frac{1}{r^{n}}\int_{B(x,2^{-m})}{2^{mn}}
\left(\frac{\mu(B^{(\alpha)}_{2^{-m}}(t, x))}{2^{-m}}\right)^{p'-1}dy\frac{dr}{r}\\
&\gtrsim \int_{0}^{\infty}
\left(\frac{\mu(B^{(\alpha)}_{r}(t, x))}{r^{n}}\right)^{p'-1}\frac{dr}{r},
\end{align*}
thereby reaching the required inequality.

\end{proof}

Now, Theorem \ref{MT} with $T_\alpha=R_\alpha$ is contained in the following result.

\begin{theorem}
\label{t21} For a nonnegative Radon measure $\mu$ on $\mathbb{R}_{+}^{1+n}$ and $\lambda>0$ set
$$
C_{R}(\mu;\lambda)=\inf\Big\{C_{p}^{(R_\alpha)}(K): \ \hbox{compact}\,\, K\subset \mathbb{R}_{+}^{1+n}\,\,\&\,\,\mu(K)\geq \lambda\Big\}.
$$
\begin{enumerate}

\item If $1<p<q<\infty$ then
$$
\eqref{eR}\Leftrightarrow\sup_{\lambda>0}\frac{\lambda^{\frac{p}{q}}}{C_{R}(\mu;\lambda)}<\infty\Leftrightarrow\sup_{(r,t_0,x_0)\in\mathbb R_+\times\mathbb R_+\times\mathbb R^n}\frac{\mu(B^{(\alpha)}_{r}(t_{0}, x_{0}))}{r^\frac{nq}{p}}<\infty.
$$

\item If $1<p=q<\infty$ then
$$
\eqref{eR}\Leftrightarrow\sup_{\lambda>0}\frac{\lambda}{C_{R}(\mu;\lambda)}<\infty\quad\left(\Rightarrow\sup_{(r,t_0,x_0)\in\mathbb R_+\times\mathbb R_+\times\mathbb R^n}\frac{\mu(B^{(\alpha)}_{r}(t_{0}, x_{0}))}{r^n}<\infty\right).
$$

\item $1<q<p<\infty$ then
$$
\eqref{eR}\Leftrightarrow\int_{0}^{\infty}\left(\frac{\lambda^{\frac{p}{q}}}{C_{R}(\mu;\lambda)}\right)^{\frac{q}{p-q}}\frac{d\lambda}{\lambda}<\infty\Leftrightarrow P_{\alpha p}^{R}\mu \in L_\mu^{q(p-1)/(p-q)}(\mathbb{R}_{+}^{1+n}).
$$
\end{enumerate}
\end{theorem}
\begin{proof} Since (1), (2) and the left equivalence of (3) are contained in \cite[Theorems 3.2\&3.3]{CX} whose proofs depend on Lemma \ref{l20}, it is enough to check the right equivalence of (3). Our approach is a fractional heat potential analogue of the Riesz potential treatment carried in \cite[Theorem 2.1]{COV}.

{\it Step 1.} We show
\begin{equation*}
\label{right} \eqref{eR}\Rightarrow P_{\alpha p}^{R}\mu \in L_\mu^{q(p-1)/(p-q)}(\mathbb{R}_{+}^{1+n}).
\end{equation*}

To do so, we first denote by $Q^{(\alpha)}_{l}$ the $\alpha$-dyadic cube with side length $l\equiv l(Q^{(\alpha)}_{l})$ and corners in the set $\{l^{2\alpha}\mathbb{Z}_{+}, l\mathbb{Z}^{n}\}$ with $\mathbb{Z}_{+}=\{0,1,2,\cdots\}$ - namely -
$$
Q^{(\alpha)}_{l}\equiv\{[k_{0}l^{2\alpha}, (k_{0}+1)l^{2\alpha})\times [k_{1}l, (k_{1}+1)l)\times\cdots\times[k_{n}l, (k_{n}+1)l)\} \,\,\,\,\hbox{as}\,\,\,\,k_{0}\in \mathbb{Z}_{+}\,\,\&\,\, k_{i}\in \mathbb{Z} 
$$
for $i=1,2,\cdots,n.$
Next, we introduce the following fractional heat Hedberg-Wolff potential generated by $\mathcal{D}^{\alpha}$ - the family of all the above-defined $\alpha$-dyadic cubes in $\mathbb R^{1+n}_+$:
$$
P_{\alpha p}^{d,R}\mu(t,x)=\sum_{Q^{(\alpha)}_{l}\in\mathcal{D}^{\alpha}}\left(\frac{\mu(Q^{(\alpha)}_{l})}{l^{n}}\right)^{p'-1}\textbf{1}_{Q^{(\alpha)}_{l}}(t,x)
$$
and then prove

\begin{equation}\label{2.5}
\eqref{eR}\Rightarrow\int_{\mathbb{R}_{+}^{1+n}}(P_{\alpha p}^{d,R}\mu(t,x))^{\frac{q(p-1)}{(p-q)}}d\mu(t,x)<\infty.
\end{equation}

Indeed, by duality, $(\ref{eR})$ is equivalent to the following inequality
$$
\|R_{\alpha}^{*}(\mathbf{g}d\mu)\|_{L^{p'}(\mathbb{R}^{n})}^{p'}\lesssim \|\mathbf{g}\|_{L_\mu^{q'}(\mathbb{R}_{+}^{1+n})}^{p'}\,\,\,\,\forall \,\,\mathbf{g}\in L_\mu^{q'=\frac{q}{q-1}}(\mathbb{R}_{+}^{1+n}).
$$
It is easy to check that Lemma \ref{l22} is also true with $P_{\alpha p}^{d,R}\mu$ in place of $P_{\alpha p}^{R}\mu$ and $\mathbf{g}d\mu$ in place of $d\mu$. So, one has
$$
\|R_{\alpha}^{*}(\mathbf{g}d\mu)\|_{L^{p'}(\mathbb{R}^{n})}^{p'}\gtrsim \int_{\mathbb{R}_{+}^{1+n}}P_{\alpha p}^{d,R}(\mathbf{g}d\mu)(t,x)\mathbf{g}(t,x)d\mu(t,x)\gtrsim \sum_{Q^{(\alpha)}_{l}}\left(\frac{\int_{Q^{(\alpha)}_{l}}\mathbf{g}(t,x)d\mu(t,x)}{l^{n}}\right)^{p'}l^{n}.
$$
Consequently,
\begin{equation}\label{2.6}
\sum_{Q^{(\alpha)}_{l}}\left(\frac{\int_{Q^{(\alpha)}_{l}}\mathbf{g}(t,x)d\mu(t,x)}{l^{n}}\right)^{p'}l^{n}\lesssim \|\mathbf{g}\|_{L_\mu^{q'}(\mathbb{R}_{+}^{1+n})}^{p'}.
\end{equation}
Upon setting
$$
\lambda_{Q^{(\alpha)}_{l}}=\left(\frac{\mu(Q^{(\alpha)}_{l})}{l^{n}}\right)^{p'}l^{n},
$$
one finds that $(\ref{2.6})$ is equivalent to
$$
\sum_{Q^{(\alpha)}_{l}}\lambda_{Q^{(\alpha)}_{l}}\left(\frac{\int_{Q^{(\alpha)}_{l}}\mathbf{g}\,d\mu}{\mu(Q^{(\alpha)}_{l})}\right)^{p'}\lesssim \|\mathbf{g}\|_{L_\mu^{q'}(\mathbb{R}_{+}^{1+n})}^{p'}.
$$
Define the following dyadic Hardy-Littlewood maximal function
$$
M_{\mu}^{d}h(t,x)=\sup_{(t,x)\in Q^{(\alpha)}}\frac{1}{\mu(Q^{(\alpha)})}\int_{Q^{(\alpha)}}|h(s,y)|d\mu(s,y)\,\,\,\,\forall\,\,Q^{(\alpha)}\in \mathcal{D}^{\alpha}.
$$
Then $M_{\mu}^{d}$ is bounded on $L^{p}_\mu(\mathbb{R}_{+}^{1+n})$ for $1<p<\infty$. Write
$$
\mathbf{g}(t,x)=(M_{\mu}^{d}h)^{\frac{1}{p'}}(t,x)
\ \hbox{under}\ 0\le h\in L_\mu^{{q'}/{p'}}(\mathbb{R}_{+}^{1+n}).
$$
It is easy to check that
$$
\left(\frac{\int_{Q^{(\alpha)}_{l}}\mathbf{g}(t,x)d\mu(t,x)}{\mu(Q^{(\alpha)}_{l})}\right)^{p'}\gtrsim \frac{\int_{Q}h(t,x)d\mu(t,x)}{\mu(Q^{(\alpha)}_{l})}
$$
and so that
$$
\|\mathbf{g}\|_{L_\mu^{q'}(\mathbb{R}_{+}^{1+n})}^{p'}\lesssim \|h\|_{L_\mu^{{q'}/{p'}}(\mathbb{R}_{+}^{1+n})}.
$$
This in turn implies
$$
\sum_{Q^{(\alpha)}_{l}}\lambda_{Q^{(\alpha)}_{l}}\frac{\int_{Q^{(\alpha)}_{l}}h(t,x)d\mu(t,x)}{\mu(Q^{(\alpha)}_{l})}\lesssim \|h\|_{L_\mu^{{q'}/{p'}}(\mathbb{R}_{+}^{1+n})}
\ \forall\ h\in L_\mu^{{q'}/{p'}}(\mathbb{R}_{+}^{1+n}),
$$
and thus via duality
$$
\sum_{Q^{(\alpha)}_{l}}\frac{\lambda_{Q^{(\alpha)}_{l}}}{\mu(Q^{(\alpha)}_{l})}\textbf{1}_{Q^{(\alpha)}_{l}}\in L_\mu^{\frac{q'}{q'-p'}}(\mathbb{R}_{+}^{1+n}),
$$
namely,
$$
\sum_{Q^{(\alpha)}_{l}}\left(\frac{\mu(Q^{(\alpha)}_{l})}{l^{n}}\right)^{p'-1}\textbf{1}_{Q^{(\alpha)}_{l}}\in L_\mu^{\frac{q(p-1)}{p-q}}(\mathbb{R}_{+}^{1+n}),
$$
which yields \eqref{2.5}.

Next, set
$$
P_{\alpha p}^{d,\tau,R}\mu(t,x)=\sum_{Q^{(\alpha)}_{l}\in\mathcal{D}^{\alpha}_{\tau}}\left(\frac{\mu(Q^{(\alpha)}_{l})}{l^{n}}\right)^{p'-1}\textbf{1}_{Q^{(\alpha)}_{l}}(t,x)\ \ \&\ \ \mathcal{D}^{\alpha}_{\tau}=\mathcal{D}^{\alpha}+\tau=\{Q^{(\alpha)'}_{l}+\tau\}_{Q^{(\alpha)'}_{l}\in \mathcal{D}^{\alpha}},
$$
where $Q^{(\alpha)}_{l}+\tau=\{(t,x)+\tau:\ (t,x)\in Q^{(\alpha)}_{l}\}$ means the $\mathbb{R}^{1+n}_+\ni\tau$-shift of $Q^{(\alpha)}_{l}$. Then \eqref{2.5} implies
\begin{equation}\label{2.7}
\sup_{\tau\in \mathbb{R}_{+}^{1+n}}\int_{\mathbb{R}_{+}^{1+n}}\big(P_{\alpha p}^{d,\tau,R}\mu(t,x)\big)^{\frac{q(p-1)}{(p-q)}}d\mu(t,x)<\infty.
\end{equation}
Now, it remains to prove
$$
P_{\alpha p}^{R}\mu \in L_\mu^{q(p-1)/(p-q)}(\mathbb{R}_{+}^{1+n}).
$$
Two situations are considered in the sequel.

{\it Case 1.1. $\mu$ is a doubling measure.} In this case, $P_{\alpha p}^{R}\mu \in L_\mu^{q(p-1)/(p-q)}(\mathbb{R}_{+}^{1+n})$ is a by-product of $(\ref{2.5})$ and the following observation
$$
P_{\alpha p}^{R}\mu(t,x)\lesssim \sum_{Q^{(\alpha)}_{l}}\left(\frac{\mu(Q^{(\alpha)*}_{l})}{l^{n}}\right)^{p'-1}\textbf{1}_{Q^{(\alpha)}_{l}}(t,x),
$$
where $Q^{(\alpha)*}_{l}$ is the cube with the same center as $Q^{(\alpha)}_{l}$ and side length two times as $Q^{(\alpha)}_{l}$.

{\it Case 1.2. $\mu$ is a possibly non-doubling measure.} For any $\rho>0$, write
$$
P_{\alpha p,\rho}^{R}\mu(t,x)=\int_{0}^{\rho}\left(\frac{\mu(B_{r}^{(\alpha)}(t,x))}{r^{n}}\right)^{p'-1}\frac{dr}{r}.
$$
Then
\begin{equation*}\label{2.8}
P_{\alpha p,\rho}^{R}\mu(t,x)\lesssim \rho^{-(n+1)}\int_{|\tau|\lesssim\rho}P_{\alpha p}^{d,\tau,R}\mu(t,x)d\tau.
\end{equation*}
In fact, for a fixed $x\in \mathbb{R}^{n}$ and $\rho>0$ with $2^{i-1}\eta\leq \rho< 2^{i}\eta$ (where $i\in\mathbb{Z}$ and $\eta>0$ will be determined later) one has
$$
P_{\alpha p,\rho}^{R}\mu(t,x)\lesssim \sum_{j=-\infty}^{i}\left(\frac{\mu(B_{2^{j}\eta}^{(\alpha)}(t,x))}{(2^{j}\eta)^{n}}\right)^{p'-1}.
$$
For $j\leq i$, let $Q^{(\alpha)}_{l,j}$ be a cube centred at $x$ with $2^{j-1}<l\leq2^{j}$. Then $B_{2^{j}\eta}^{(\alpha)}(t,x)\subseteq Q^{(\alpha)}_{l,j}$ for sufficiently small $\eta$. Assume not only that $E$ is the set of all points $\tau\in \mathbb{R}_{+}^{1+n}$ enjoying $|\tau|\lesssim\rho$ with $|E|$ being the $(1+n)$-dimensional Lebesgue measure, but also that there exists $Q^{(\alpha),\tau}_{l}\in \mathcal{D}^{\alpha}_{\tau}$ satisfying $l=2^{j+1}$ and $Q^{(\alpha)}_{l,j}\subseteq Q^{(\alpha),\tau}_{l}$. A geometric consideration produces a dimensional constant $c(n)>0$ such that $|E|\ge c(n)\rho^{n+1}\ \forall\ j\leq i$. Consequently, one has
\begin{align*}
\mu(B_{2^{j}\eta}^{(\alpha)}(t,x))^{p'-1}&\lesssim |E|^{-1}\int_{E}\sum_{l=2^{j+1}}\mu(Q^{(\alpha),\tau}_{l})^{p'-1}\textbf{1}_{Q^{(\alpha),\tau}_{l}}(t,x)d\tau\\
&\lesssim \rho^{-(n+1)}\int_{|\tau|\lesssim \rho}\sum_{l=2^{j+1}}\mu(Q^{(\alpha),\tau}_{l})^{p'-1}\textbf{1}_{Q^{(\alpha),\tau}_{l}}(t,x)d\tau,
\end{align*}
and so that
\begin{align*}
P_{\alpha p,\rho}^{R}\mu(t,x)&\lesssim \rho^{-(n+1)}\int_{|\tau|\lesssim\rho} \sum_{j=-\infty}^{i}\sum_{l=2^{j+1}}\left(\frac{\mu(Q^{(\alpha),\tau}_{l})}{(2^{j}\eta)^{n}}\right)^{p'-1}\textbf{1}_{Q^{(\alpha),\tau}_{l}}(t,x)ds\\
&\lesssim \rho^{-(n+1)}\int_{|\tau|\lesssim\rho}P_{\alpha p}^{d,\tau,R}\mu(t,x)d\tau,
\end{align*}
whence reaching $(\ref{2.8})$.

From $(\ref{2.8})$, the H\"{o}lder inequality and Fubini's theorem it follows that
\begin{align*}
&\int_{\mathbb{R}_{+}^{1+n}}\Big(P_{\alpha p,\rho}^{R}\mu(t,x)\Big)^{\frac{q(p-1)}{p-q}}d\mu(t,x)\\
&\ \ \lesssim \int_{\mathbb{R}_{+}^{1+n}} \left(\rho^{-(n+1)}\left(\int_{|\tau|\leq C\rho} \Big(P_{\alpha p}^{d,\tau,R}\mu\Big)^{\frac{q(p-1)}{p-q}}d\tau\right)^{\frac{p-q}{q(p-1)}}\left(\int_{|\tau|\lesssim\rho}d\tau\right)^{1-\frac{p-q}{q(p-1)}}\right)^{\frac{q(p-1)}{p-q}}d\mu\\
&\ \ \lesssim\rho^{-(n+1)}\int_{|\tau|\lesssim\rho}\left( \int_{\mathbb{R}_{+}^{1+n}}\Big(P_{\alpha p}^{d,\tau,R}\mu\Big)^{\frac{q(p-1)}{p-q}}d\mu\right) d\tau\\
&\ \ \le \kappa(n),
\end{align*}
where the last constant $\kappa(n)$ is independent of $\rho$. This clearly produces
$$
P_{\alpha p}^{R}\mu \in L_\mu^{q(p-1)/(p-q)}(\mathbb{R}_{+}^{1+n})
$$
via letting $\rho\rightarrow \infty$ and utilizing the monotone convergence theorem.

{\it Step 2.} We prove
$$
P_{\alpha p}^{R}\mu \in L_\mu^{q(p-1)/(p-q)}(\mathbb{R}_{+}^{1+n})\Rightarrow\eqref{eR}.
$$

Recall that \eqref{eR} is equivalent to the following inequality
$$
\|R_{\alpha}^{*}(\mathbf{g}d\mu)\|_{L^{p'}(\mathbb{R}^{n})}\lesssim \|\mathbf{g}\|_{L_\mu^{q'}(\mathbb{R}_{+}^{1+n})}\,\,\,\,\forall\,\,\mathbf{g}\in L_\mu^{q'}(\mathbb{R}_{+}^{1+n}).
$$
Thus, by Lemma \ref{l22}, it is sufficient to check that $P_{\alpha p}^{R}\mu \in L_\mu^{q(p-1)/(p-q)}(\mathbb{R}_{+}^{1+n})$ implies
\begin{equation}\label{2.9}
\int_{\mathbb{R}_{+}^{1+n}}P_{\alpha p}^{R}(\mathbf{g}d\mu)(t,x)\mathbf{g}(t,x)d\mu\lesssim \|\mathbf{g}\|_{L^{q'}(\mathbb{R}_{+}^{1+n},\,d\mu)}^{p'}\,\,\,\,\forall\,\,\mathbf{g}\in L_\mu^{q'}(\mathbb{R}_{+}^{1+n}).
\end{equation}
There is no loss of generality in assuming $\mathbf{g}\geq 0$. Since
\begin{align*}
P_{\alpha p}^{R}(\mathbf{g}d\mu)(t,x)
&\approx\int_{0}^{\infty} \left(\frac{\mu(B_{r}^{(\alpha)}(t,x))}{r^{n}}\right)^{p'-1}\left(\frac{\int_{B_{r}^{(\alpha)}(t,x)}\mathbf{g}(t,x)d\mu}{\mu(B_{r}^{(\alpha)}(t,x))}\right)^{p'-1}\frac{dr}{r}\\
&\lesssim \left(M_{\mu}\mathbf{g}(t,x)\right)^{p'-1}P_{\alpha p}^{R}\mu(t,x),
\end{align*}
an application of the H\"{o}lder inequality gives
\begin{align*}
&\int_{\mathbb{R}_{+}^{1+n}}P_{\alpha p}^{R}(\mathbf{g}d\mu)(t,x)d\mu(t,x)\\
&\ \ \lesssim \int_{\mathbb{R}_{+}^{1+n}}\left(M_{\mu}\mathbf{g}(t,x)\right)^{p'-1}P_{\alpha p}^{R}\mu(t,x)\mathbf{g}(t,x)d\mu(t,x)\\
&\ \ \lesssim \left(\int_{\mathbb{R}_{+}^{1+n}}\left(M_{\mu}\mathbf{g}(t,x)\right)^{q'}d\mu(t,x)\right)^{\frac{q'}{p'-1}}\left(\int_{\mathbb{R}_{+}^{1+n}}\Big(\mathbf{g}(t,x)P_{\alpha p}^{R}\mu(t,x)\Big)^{\frac{q'}{q'-p'+1}}d\mu(t,x)\right)^{\frac{q'-p'+1}{q'}}.
\end{align*}
Here
$$
M_{\mu}\mathbf{g}(t,x)=\sup_{r>0}\frac{1}{\mu(B_{r}^{(\alpha)}(t,x))}\int_{B_{r}^{(\alpha)}(t,x)}\mathbf{g}(s,y)d\mu(s,y)
$$
is the centered Hardy-Littlewood maximal function of $\mathbf{g}$ with respect to $\mu$. The fact that $M_{\mu}$ is bounded on $L_\mu^{q'}(\mathbb{R}_{+}^{1+n})$ (cf. \cite{F}) and H\"{o}lder's inequality imply
$$
\int_{\mathbb{R}_{+}^{1+n}}P_{\alpha p}^{R}(\mathbf{g}d\mu)(t,x)d\mu(t,x)\lesssim \|\mathbf{g}\|_{L_\mu^{q'}(\mathbb{R}_{+}^{1+n})}^{p'}\left(\int_{\mathbb{R}_{+}^{1+n}}\Big(P_{\alpha p}^{R}\mu\Big)^{\frac{q(p-1)}{p-q}}d\mu\right)^{\frac{p-q}{q(p-1)}},
$$
whence \eqref{2.9}.
\end{proof}

\section{$S_{\alpha}$'s tracing}\label{section3}

In this section we verify Theorem \ref{MT} for $T_\alpha=S_\alpha$ and $1<p<1+\frac{n}{2\alpha}$. Like proving Theorem \ref{MT} for $T_\alpha=R_\alpha$, three lemmas are required in what follows.

The first is regarding the dual formulation of $C_p^{(S_\alpha)}(K)$ of a given compact set $K\subset\mathbb R^{1+n}_+$.

\begin{lemma}
\label{l31}
Let $\mu\in \mathcal{U}^{+}(K)$, $1<p<1+\frac{n}{2\alpha}$, $p'=\frac{p}{p-1}$, $S^\ast_\alpha$ be the adjoint operator of $S_\alpha$, and
$$
P^S_{\alpha p}\mu(t,x)=\int_0^\infty\left(\frac{\mu\big(B^{(\alpha)}_r(t,x)\big)}{r^{n+2\alpha(1-p)}}\right)^{p'-1}\frac{dr}{r}\quad\forall\quad (t,x)\in\mathbb R^{1+n}_+.
$$
Then:

$(a)$
$$
C_{p}^{(S_\alpha)}(K)=\sup\{\|\mu\|^{p}: \mu\in \mathcal{U}^{+}(K)\,\,\&\,\,\|S_{\alpha}^{*}\mu\|_{L^{p'}(\mathbb{R}_{+}^{1+n})}\leq 1\}.
$$

$(b)$
$$
\|S_{\alpha}^{*}\mu\|_{L^{p'}(\mathbb{R}^{1+n}_+)}^{p'}\approx \int_{\mathbb{R}_{+}^{1+n}}P_{\alpha p}^{S}\mu(t,x)d\mu(t,x).
$$
\end{lemma}
\begin{proof} $(a)$ Since $S_\alpha^\ast$ is determined by
$$
\int_{\mathbb{R}_{+}^{1+n}}(S_{\alpha}g) h\,dtdx=\int_{\mathbb{R}_{+}^{1+n}}g(t,x)\left(\int_{t}^{\infty}e^{-(s-t)(-\Delta_x)^{\alpha}}h(s,x)ds\right)dtdx\ \forall\ g,h\in C_{0}^{\infty}(\mathbb{R}_{+}^{1+n}),
$$
it follows that for any $h \in C_{0}^{\infty}(\mathbb{R}_{+}^{1+n})$ one has
$$
S_{\alpha}^{*}h(t,x)=\int_{t}^{\infty}e^{-(s-t)(-\Delta_x)^{\alpha}}h(s,x)ds=\int_{[t,\infty)\times\mathbb R^n}K^{(\alpha)}_{s-t}(x-y)h(s,t)\,dsdy.
$$
The definition of $S_{\alpha}^{*}$ is extended to the family of all Borel measures $\mu$ with compact support in $\mathbb R^{1+n}_+$:
$$
S_{\alpha}^{*}\mu(t,x)=\int_{[t,\infty)\times{\mathbb R^n}}K^{(\alpha)}_{s-t}(x-y)\,d\mu(s,y).
$$
According to \cite[Propostion 2.1]{JXYZ}, we have
$$
C_{p}^{(S_\alpha)}(K)=\sup\{\|\mu\|^{p}:\ \mu\in \mathcal{U}^{+}(K)\,\,\&\,\,\|S_{\alpha}^{*}\mu\|_{L^{p'}(\mathbb{R}_{+}^{1+n})}\leq 1\}.
$$

$(b)$  This can be proved via a slight modification of the argument for Lemma \ref{l22} - in particular - via replacing the maximal function $M_{\alpha}\mu(x)$ by
$$
M_{\alpha}\mu(t,x)=\sup_{r>0}{r^{-n}}\int_{B_{r}^{(\alpha)}(t,x)}d\mu.
$$
\end{proof}

The second indicates that $C_p^{(S_\alpha)}(K)$ of a given compact $K\subset\mathbb R^{1+n}_+$ can be realized by $\mu_K(K)$ of an element $\mu_K\in\mathcal{U}^+(K)$.

\begin{lemma}\label{Lemma 3.1} Let $K$ be a compact subset of $\mathbb{R}_{+}^{1+n}$. Then there exists a $\mu_{K}\in \mathcal{U}^{+}(K)$ such that
$$
\mu_{K}(K)=\int_{\mathbb{R}_{+}^{1+n}}(S_{\alpha}^{*}\mu_{K}(t,x))^{p'}dtdx=\int_{\mathbb{R}_{+}^{1+n}}S_{\alpha}(S_{\alpha}^{*}\mu_{K})^{p'-1}d\mu_{K}=C_{p}^{(S_\alpha)}(K).
$$
\end{lemma}

\begin{proof} Lemma \ref{l31}(a) (plus \cite[Propostion 2.1]{JXYZ}) ensures the existence of a sequence $\{\mu_{i}\}\subset \mathcal{U}^{+}(K)$ such that
$$
\|S_{\alpha}^{*}\mu_{i}\|_{L^{p'}(\mathbb{R}_{+}^{1+n})}\le 1\,\,\,\,\&\,\,\,\,\lim_{i\rightarrow \infty}\mu_{i}(K)=\big(C_{p}^{(S_\alpha)}(K)\big)^\frac1p
$$
and $\mu_{i}$ has a weak limit $\mu\in \mathcal{U}^{+}(K)$. Thus $\mu(K)=\big(C_{p}^{(S_\alpha)}(K)\big)^\frac1p$. It follows from the lower semi-continuity of $S_{\alpha}^{*}\mu$ on $\mathcal{U}^{+}(K)$ that $\|S_{\alpha}^{*}\mu\|_{L^{p'}(\mathbb{R}_{+}^{1+n})}\leq 1$.
Meanwhile, the following estimation
$$
\|\mu\|\leq \int_{\mathbb{R}_{+}^{1+n}}S_{\alpha}gd\mu=\int_{\mathbb{R}_{+}^{1+n}}g(t,x)S_{\alpha}^{*}\mu(t,x)dtdx\leq \|g\|_{L^{p}(\mathbb{R}_{+}^{1+n})}\|S_{\alpha}^{*}\mu\|_{L^{p'}(\mathbb{R}_{+}^{1+n})}
$$
gives $\|S_{\alpha}^{*}\mu\|_{L^{p'}(\mathbb{R}_{+}^{1+n})}\geq 1$.
So, $\|S_{\alpha}^{*}\mu\|_{L^{p'}(\mathbb{R}_{+}^{1+n})}= 1$.

Choosing $\mu_{K}=C_{p}^{(S_\alpha)}(K)^{\frac{1}{p'}}\mu$ and using $\mu(K)=\big(C_p^{(S_\alpha)}(K)\big)^\frac1p$, one has
$$
\mu_{K}(K)=\int_{\mathbb{R}_{+}^{1+n}}(S_{\alpha}^{*}\mu_{K}(t,x))^{p'}dtdx=C_{p}^{(S_\alpha)}(K).
$$

Suppose that $g_{0}$ is the capacitary potential of $C_{p}^{(S_\alpha)}(K)$, i.e.,
$$
\|g_{0}\|_{L^{p}(\mathbb{R}_{+}^{1+n})}^{p}=C_{p}^{(S_\alpha)}(K)\,\,\&\,\,S_{\alpha} g_{0}\geq\textbf{1}_{K}.
$$
Then $g_{0}(t,x)=(S_{\alpha}^{*}\mu_{K})^{p'-1}(t,x)$. A further use of \cite[Propostion 2.1]{JXYZ} derives
$$
\mu_{K}(\{(t,x)\in K: S_{\alpha}g_{0}(t,x)<1\})=0,
$$
whence
$$
S_{\alpha}(g_{0})=S_{\alpha}(S_{\alpha}^{*}\mu_{K})^{p'-1}\geq 1\,\,\,\,\hbox{a.e.} \,\,\mu_{K}\,\,\hbox{on}\,\,K.
$$

Now, Fubini's theorem and the H\"{o}lder inequality are utilized to derive
\begin{align*}
C_{p}^{(S_\alpha)}(K)&\leq \int_{\mathbb{R}_{+}^{1+n}}S_{\alpha}g_{0}d\mu_{K}\\
&= \int_{\mathbb{R}_{+}^{1+n}}\int_{0}^{t}\int_{\mathbb{R}^{n}}K_{t-s}^{(\alpha)}(x-y)f(s,y)dydsd\mu_{K}\\
&=\int_{\mathbb{R}^{n}}\int_{0}^{t}\int_{s}^{\infty}\int_{\mathbb{R}^{n}}K_{t-s}^{(\alpha)}(x-y)d\mu_{K}f(s,y)dsdy\\
&\leq \int_{\mathbb{R}_{+}^{1+n}}S_{\alpha}^{*}\mu_{K}(t,x)g_{0}(t,x)dtdx\\
&\leq \|S_{\alpha}^{*}\mu_{K}\|_{L^{p'}(\mathbb{R}_{+}^{1+n})}\|g_{0}\|_{L^{p}(\mathbb{R}_{+}^{1+n})}\\
&=C_{p}^{(S_\alpha)}(K),
\end{align*}
thereby completing the proof.
\end{proof}

The third is concerning the weak and strong type estimates for $C_p^{(S_\alpha)}$.

\begin{lemma}\label{Lemma 3.2} Let $1<p<\infty$ and $L^{p}_{+}(\mathbb{R}_{+}^{1+n})$ stand for the class of all nonnegative functions in $L^{p}(\mathbb{R}_{+}^{1+n})$. If $g\in L^{p}_{+}(\mathbb{R}_{+}^{1+n})$ and $\lambda>0$, then:

$(a)$ $C_{p}^{(S_\alpha)}(\{(t,x)\in \mathbb{R}_{+}^{1+n}: S_{\alpha}g(t,x)\geq \lambda\})\leq \lambda^{-p}\|g\|_{L^{p}(\mathbb{R}_{+}^{1+n})}^{p}$;

$(b)$ $\int_{0}^{\infty}C_{p}^{(S_\alpha)}(\{(t,x)\in \mathbb{R}_{+}^{1+n}: S_{\alpha}g(t,x)\geq \lambda\})d\lambda^{p}\lesssim \|g\|_{L^{p}(\mathbb{R}_{+}^{1+n})}^{p}.$
\end{lemma}
\begin{proof} $(a)$ This follows immediately from the definition of $C_{p}^{(S_\alpha)}$.

$(b)$ It is enough to check this inequality for any nonnegative function $g\in C_{0}^{\infty}(\mathbb{R}_{+}^{1+n})$. The forthcoming demonstration is a slight modification of the argument for \cite[Lemma 3.1]{CX}.

For each $i=0,\pm 1, \pm 2,\cdots$ and any nonnegative function $g\in C_{0}^{\infty}(\mathbb{R}_{+}^{1+n})$, we follow the proof of \cite[Theorem 7.1.1]{AH} to write
$$
K_{i}=\{(t,x)\in \mathbb{R}_{+}^{1+n}: S_{\alpha}g(t,x)\geq 2^{i}\}.
$$
Assume that $\mu_{i}$ is the measure obtained in Lemma \ref{Lemma 3.1} for $K_{i}$. Then by duality and H\"{o}lder's inequality, one has
\begin{align*}
\sum_{i=-\infty}^{\infty}2^{i p}\mu_{i}(\mathbb{R}_{+}^{1+n})&\leq \sum_{i=-\infty}^{\infty}2^{i (p-1)}\int_{\mathbb{R}_{+}^{1+n}}g(t,x)S_{\alpha}^{*}\mu_{i}(t,x)dtdx\\
&\lesssim \|g\|_{L^{p}(\mathbb{R}_{+}^{1+n})}\left\|\sum_{i=-\infty}^{\infty}2^{i (p-1)}S_{\alpha}^{*}\mu_{i}\right\|_{L^{p'}(\mathbb{R}_{+}^{1+n})}\\
&\equiv\|g\|_{L^{p}(\mathbb{R}_{+}^{1+n})}\|\eta\|_{L^{p'}(\mathbb{R}_{+}^{1+n})},
\end{align*}
where
$$
\eta(t,x)=\sum_{i=-\infty}^{\infty}2^{i (p-1)}S_{\alpha}^{*}\mu_{i}(t,x).
$$
For $k=0, \pm 1, \pm 2, \cdots,$ let
$$\eta_{k}(t,x)=\sum_{i=-\infty}^{k}2^{i(p-1)}S_{\alpha}^{*}\mu_{i}(t,x).$$
Then it is easy to find that
$$
\eta_{k}\in L^{p'}(\mathbb{R}_{+}^{1+n})\ \&\ \lim_{k\rightarrow \infty}\eta_{k}=\eta\ \hbox{in}\ L^{p'}(\mathbb R^{1+n}_+).
$$
We next to prove that
\begin{equation}\label{3.3}
\|\eta\|_{L^{p'}(\mathbb{R}_{+}^{1+n})}^{p'}\lesssim \sum_{i=-\infty}^{\infty}2^{i p}\mu_{i}(\mathbb{R}_{+}^{1+n})
\end{equation}
according to two cases.

{\it Case: $2<p<\infty$}. Notice first that
\begin{equation}\label{3.4}
\eta(t,x)^{p'}=p'\sum_{k=-\infty}^{\infty}\eta_{k}(t,x)^{p'-1}2^{k(p-1)}S_{\alpha}^{*}\mu_{k}(t,x)\,\,\,\,\hbox{a.e.}\,\,(t,x)\in \mathbb{R}_{+}^{1+n}.
\end{equation}
So, the H\"{o}lder inequality yields that
$$
\|\eta\|_{L^{p'}(\mathbb{R}_{+}^{1+n})}^{p'}\lesssim\frac{ \left(\int_{\mathbb{R}_{+}^{1+n}}\sum_{k=-\infty}^{\infty}2^{kp}(S_{\alpha}^{*}\mu_{k})^{p'}(t,x)dtdx\right)^{2-p'}}{ \left(\int_{\mathbb{R}_{+}^{1+n}}\sum_{k=-\infty}^{\infty}\eta_{k}(t,x)2^{k}(S_{\alpha}^{*}\mu_{k})^{p'-1}(t,x)dtdx\right)^{1-p'}}.
$$
Since Lemma \ref{Lemma 3.1} gives
$$
\int_{\mathbb{R}_{+}^{1+n}}\sum_{k=-\infty}^{\infty}2^{kp}(S_{\alpha}^{*}\mu_{k})^{p'}(t,x)dtdx=\sum_{k=-\infty}^{\infty}2^{kp}\int_{\mathbb{R}_{+}^{1+n}}(S_{\alpha}^{*}\mu_{k})^{p'}(t,x)dtdx
=\sum_{k=-\infty}^{\infty}2^{kp}\mu_{k}(\mathbb{R}_{+}^{1+n})
$$
and
\begin{align*}
&\int_{\mathbb{R}_{+}^{1+n}}\sum_{k=-\infty}^{\infty}\eta_{k}(t,x)2^{k}(S_{\alpha}^{*}\mu_{k})^{p'-1}(t,x)dtdx\\
&\ \ =
\sum_{k=-\infty}^{\infty}\sum_{i\leq k}2^{i(p-1)+k}\int_{\mathbb{R}_{+}^{1+n}}(S_{\alpha}^{*}\mu_{i}(t,x))(S_{\alpha}^{*}\mu_{k}(t,x))^{p'-1}dtdx\\
&\ \ \lesssim \sum_{k=-\infty}^{\infty}2^{kp}C_{p}^{(S_\alpha)}(K_{k})\\
&\ \ \approx \sum_{k=-\infty}^{\infty}2^{kp}\mu_{k}(\mathbb{R}_{+}^{1+n}),
\end{align*}
$(\ref{3.3})$ is true for $2<p<\infty$.

{\it Case: $1<p\leq 2$.} A combination of $(\ref{3.4})$ and Minkowski's inequality gives
\begin{align*}
\|\eta\|_{L^{p'}(\mathbb{R}_{+}^{1+n})}^{p'}&=\sum_{k=-\infty}^{\infty}2^{k(p-1)}
\int_{\mathbb{R}_{+}^{1+n}}\left(\sum_{i=-\infty}^{k}2^{i(p-1)}S_{\alpha}^{*}\mu_{i}(t,x)\right)^{p'-1}(S_{\alpha}^{*}\mu_{k}(t,x))dtdx\\
&\lesssim \sum_{k=-\infty}^{\infty}2^{k(p-1)}
\left(\sum_{i=-\infty}^{k}2^{i(p-1)}\left(\int_{\mathbb{R}_{+}^{1+n}}\big(S_{\alpha}^{*}\mu_{i}(t,x)\big)^{p'-1}S_{\alpha}^{*}\mu_{i}(t,x)dtdx\right)^{\frac{1}{p'-1}}\right)^{p'-1}\\
&\approx \sum_{k=-\infty}^{\infty}2^{k(p-1)}\left(\sum_{i=-\infty}^{k}2^{i(p-1)}C_{p}^{(S_\alpha)}(K_{i})^{\frac{1}{p'-1}}\right)^{p'-1}\\
&\lesssim \sum_{k=-\infty}^{\infty}2^{k(p-1)}C_{p}^{(S_\alpha)}(K_{k})\\
&\lesssim \sum_{k=-\infty}^{\infty}2^{k(p-1)}\mu_{k}(\mathbb{R}_{+}^{1+n}),
\end{align*}
whence yields $(\ref{3.3})$ under $1<p\le 2$.

As a consequence, \eqref{3.3} plus
$$
\sum_{i=-\infty}^{\infty}2^{ip}\mu_{i}(\mathbb{R}_{+}^{1+n})\lesssim \sum_{i=-\infty}^{\infty}2^{ip}C_{p}^{(S_\alpha)}(K_{i})\lesssim \|g\|_{L^{p}(\mathbb{R}_{+}^{1+n})}^{p},
$$
implies the desired inequality in $(b)$.
\end{proof}

Now, Theorem \ref{MT} for $T_\alpha=S_\alpha$ is contained in the following assertion.

\begin{theorem}
\label{t31}
For a nonnegative Radon measure $\mu$ on $\mathbb R^{1+n}_+$ and $\lambda>0$ set
$$
C_{S}(\mu;\lambda)=\inf\Big\{C_{p}^{(S_\alpha)}(K):\ \hbox{compact}\,\, K\subset \mathbb{R}_{+}^{1+n}\,\,\&\,\,\mu(K)\geq \lambda\Big\}.
$$

\begin{enumerate}

\item If $1<p<\min\{q,1+\frac{n}{2\alpha}\}$ then
$$
\eqref{eS}\Leftrightarrow\sup_{\lambda>0}\frac{\lambda^{\frac{p}{q}}}{C_{S}(\mu;\lambda)}<\infty\Leftrightarrow\sup_{(r,t_0,x_0)\in\mathbb R_+\times\mathbb R_+\times\mathbb R^n}\frac{\mu(B^{(\alpha)}_{r}(t_{0}, x_{0}))}{r^\frac{(n+2\alpha(1-p))q}{p}}<\infty.
$$

\item If $1<p=q<1+\frac{n}{2\alpha}$ then
$$
\eqref{eS}\Leftrightarrow\sup_{\lambda>0}\frac{\lambda}{C_{S}(\mu;\lambda)}<\infty\quad\left(\Rightarrow\sup_{(r,t_0,x_0)\in\mathbb R_+\times\mathbb R_+\times\mathbb R^n}\frac{\mu(B^{(\alpha)}_{r}(t_{0}, x_{0}))}{r^{n+2\alpha(1-p)}}<\infty\right).
$$

\item $1<q<p<1+\frac{n}{2\alpha}$ then
$$
\eqref{eS}\Leftrightarrow\int_{0}^{\infty}\left(\frac{\lambda^{\frac{p}{q}}}{C_{S}(\mu;\lambda)}\right)^{\frac{q}{p-q}}\frac{d\lambda}{\lambda}<\infty\Leftrightarrow P_{\alpha p}^{S}\mu \in L_\mu^{q(p-1)/(p-q)}(\mathbb{R}_{+}^{1+n}).
$$
\end{enumerate}
\end{theorem}
\begin{proof} (1) Suppose \eqref{eS} is valid. Then, for a given compact set $K\subset \mathbb{R}_{+}^{1+n}$, an application of Lemma \ref{Lemma 3.1} and  H\"{o}lder's inequality gives
$$
\int_{\mathbb{R}_{+}^{1+n}}gS_{\alpha}^{*}\mu_{K}dtdx=\int_{\mathbb{R}_{+}^{1+n}}S_{\alpha} gd\mu_{K}\leq \|S_{\alpha} g\|_{L_\mu^{q}(\mathbb{R}_{+}^{1+n})}\mu(K)^{\frac{1}{q'}}\lesssim \|g\|_{L^{p}(\mathbb{R}_{+}^{1+n})}\mu(K)^{\frac{1}{q'}},
$$
whence
$$
\|S_{\alpha}^{*}\mu_{K}\|_{L^{p'}(\mathbb{R}_{+}^{1+n})}\lesssim \mu(K)^{\frac{1}{q'}}.
$$
This shows that for
$$
E_{\lambda}(g)\equiv\Big\{(t,x)\in \mathbb{R}_{+}^{1+n}:|S_{\alpha} g(t,x)|\geq \lambda\Big\}\quad\forall\quad\lambda>0
$$
one has
\begin{align*}
\lambda\mu(E_{\lambda}(g))&\leq \int_{\mathbb{R}_{+}^{1+n}}|S_{\alpha} g|d\mu_{E_{\lambda}}\\
&\lesssim
\|g\|_{L^{p}(\mathbb{R}_{+}^{1+n})}\|S_{\alpha}^{*}\mu_{E_{\lambda}}\|_{L^{p'}(\mathbb{R}_{+}^{1+n})}\\
&\lesssim \|g\|_{L^{p}(\mathbb{R}_{+}^{1+n})}\mu(E_{\lambda})^{\frac{1}{q'}}.
\end{align*}
Therefore, we obtain
$$
\sup_{\lambda>0}\lambda^{q}\mu(E_{\lambda}(g))\lesssim \|g\|_{L^{p}(\mathbb{R}_{+}^{1+n})}^{q}.
$$
Picking a function $g\in L^{p}(\mathbb{R}_{+}^{1+n})$ such that $S_{\alpha} g\geq 1$ on a given compact $K\subset \mathbb{R}_{+}^{1+n}$, we conclude that
$$
\mu(K)^{\frac{1}{q}}\lesssim C_{p}^{(S_\alpha)}(K)^{\frac{1}{p}}\ \hbox{and\ hence}\ \lambda^{\frac{1}{q}}\lesssim C_{S}(\mu;\lambda)^{\frac{1}{p}}\ \forall\ \lambda>0.
$$

Conversely, if the last inequality is valid, then
$$
\mu(K)^{\frac{1}{q}}\lesssim C_{p}^{(S_\alpha)}(K)^{\frac{1}{p}}\,\,\,\,\forall \,\,\hbox{compact}\,\, K\subset \mathbb{R}_{+}^{1+n}.
$$
Lemma \ref{Lemma 3.2} is used to derive that if $g\in L^p(\mathbb R^{1+n}_+)$ then
\begin{align*}
\int_{\mathbb{R}_{+}^{1+n}}|S_{\alpha} g|^{q}d\mu&=\int_{0}^{\infty}\mu(E_{\lambda})d\lambda^{q}\\
&\lesssim \int_{0}^{\infty}C_{p}^{(S_\alpha)}(E_{\lambda})^{\frac{q-p}{p}}C_{p}^{(S_\alpha)}(E_{\lambda})\lambda^{q-p}d\lambda^{p}\\
&\lesssim \|g\|_{L^{p}(\mathbb{R}_{+}^{1+n})}^{q-p}\int_{0}^{\infty}C_{p}^{(S_\alpha)}(E_{\lambda})d\lambda^{p}\\
&\lesssim \|g\|_{L^{p}(\mathbb{R}_{+}^{1+n})}^{q}.
\end{align*}
Namely, \eqref{eS} holds.

Next, an application of \eqref{capBall} derives that
$$
\lambda^{\frac{1}{q}}\lesssim C_{S}(\mu;\lambda)^{\frac{1}{p}}\ \forall\ \lambda>0\Rightarrow
\mu(B_{r}^{(\alpha)}(t_{0},x_{0}))\lesssim r^{\frac{q}{p}(n+2\alpha-2\alpha p)}\,\,\,\,\forall\ r>0.
$$
For the reverse implication, we first note that $(t,x)\in B_{r}^{(\alpha)}(t_{0},x_{0})$ ensures $K_{t-t_{0}}^{(\alpha)}(x-x_{0})\gtrsim r^{-n}$. This, along with Fubini's theorem, yields
\begin{align*}
S_{\alpha}^{*}\mu_{K}(t_{0},x_{0})&\approx \int_{t_{0}}^{\infty}\int_{\mathbb{R}^{n}}\left(\int_{(K_{t-t_{0}}^{(\alpha)}(x-x_{0}))^{-\frac{1}{n}}}^{\infty}\frac{dr}{r^{n+1}}\right)d\mu_{K}\\
&\lesssim \int_{t_{0}}^{\infty}\int_{\mathbb{R}^{n}}\left(\int_{0}^{\infty}\textbf{1}_{B_{r}^{(\alpha)}(t_{0},x_{0})}\frac{dr}{r^{n+1}}\right)d\mu_{K}\\
&\lesssim \int_{0}^{\infty}\mu_{K}(B_{r}^{(\alpha)}(t_{0},x_{0}))\frac{dr}{r^{n+1}}.
\end{align*}
Therefore, for a $\delta>0$ to be determined later, we use the Minkwoski inequality to get
\begin{align*}
\|S_{\alpha}^{*}\mu_{K}\|_{L^{p'}(\mathbb{R}_{+}^{1+n})}&\lesssim \int_{\mathbb{R}_{+}^{1+n}}\left(\int_{0}^{\infty}\mu_{K}(B_{r}^{(\alpha)}(t_{0},x_{0}))\frac{dr}{r^{n+1}}\right)^{p'}dtdx\\
&\lesssim \int_{0}^{\infty}\|\mu_{K}(B_{r}^{(\alpha)}(\cdot,\cdot))\|_{L^{p'}(\mathbb{R}_{+}^{1+n})}\frac{dr}{r^{n+1}}\\
&=\int_{0}^{\delta}\|\mu_{K}(B_{r}^{(\alpha)}(\cdot,\cdot))\|_{L^{p'}(\mathbb{R}_{+}^{1+n})}\frac{dr}{r^{n+1}}\\
&\quad+
\int_{\delta}^{\infty}\|\mu_{K}(B_{r}^{(\alpha)}(\cdot,\cdot))\|_{L^{p'}(\mathbb{R}_{+}^{1+n})}\frac{dr}{r^{n+1}}\\
&\equiv I_{1}+I_{2}.
\end{align*}
Since
$$
\|\mu_{K}(B_{r}^{(\alpha)}(\cdot,\cdot))\|_{L^{p'}(\mathbb{R}_{+}^{1+n})}^{p'}\lesssim \mu(K)^{p'-1}\int_{\mathbb{R}_{+}^{1+n}}\mu_{K}(B_{r}^{(\alpha)}(t,x))dtdx\lesssim \mu(K)^{p'-1}r^{n+2\alpha},
$$
it follows that
$$I_{2}\lesssim \mu(K)\int_{\delta}^{\infty}\frac{dr}{r^{n+1-\frac{n+2\alpha}{p'}}}\lesssim \mu(K)\delta^{2\alpha-\frac{n+2\alpha}{p}}.$$
On the other hand,
$$
\mu(B_{r}^{(\alpha)}(t_{0},x_{0}))\lesssim r^{\frac{q}{p}(n+2\alpha-2\alpha p)}\ \forall\ r>0
$$
derives
\begin{align*}
\|\mu_{K}(B_{r}^{(\alpha)}(\cdot,\cdot))\|_{L^{p'}(\mathbb{R}_{+}^{1+n})}^{p'}&\lesssim
r^{\frac{q(n+2\alpha-2\alpha p)(p'-1)}{p}}\int_{\mathbb{R}_{+}^{1+n}}\mu_{K}(B_{r}^{(\alpha)}(t,x))dtdx\\
&\lesssim \mu(K)r^{\frac{q(n+2\alpha-2\alpha p)(p'-1)}{p}+n+2\alpha}.
\end{align*}
This clearly forces
$$
I_{1}\lesssim \mu({K})^{\frac{1}{p'}}\int_{0}^{\delta}r^{(p')^{-1}\big(\frac{q(n+2\alpha-2\alpha p)(p'-1)}{p}+n+2\alpha\big)}\frac{dr}{r^{n+1}}\lesssim
\mu(K)^{\frac{1}{p'}}\delta^{\frac{(q-p)(n+2\alpha-2\alpha p)}{p^{2}}}.
$$
Upon choosing $\delta=\mu{(K)}^{\frac{p}{q(n+2\alpha-2\alpha p)}}$, we obtain
$$
\|S_{\alpha}^{*}\mu_{K}\|_{L^{p'}(\mathbb{R}_{+}^{1+n})} \lesssim \mu(K)^{\frac{1}{q'}}\ \hbox{and\ hence}\  C_{p}^{(S_\alpha)}(K)^{\frac{1}{p'}}\lesssim \mu(K)^{\frac{1}{q'}}.
$$

(2) This follows from the above demonstration.

(3) Suppose \eqref{eS} is valid. Then
$$
\sup_{\lambda>0}\lambda(\mu(E_{\lambda}(g)))^{\frac{1}{q}}\lesssim \|g\|_{L^{p}(\mathbb{R}_{+}^{1+n})}\,\,\,\forall\,\,g \in L^{p}(\mathbb{R}_{+}^{1+n}).
$$
For each integer $i\in\mathbb Z$, there is a compact set $K_i\subset \mathbb{R}_{+}^{1+n}$ and a function $g_{i}\in L^{p}(\mathbb{R}_{+}^{1+n})$ such that
$$
C_{p}^{(S_\alpha)}(K_{i})\lesssim C_{S}(\mu; 2^{i}), \,\,\mu(K_{i})>2^{i};\,\,S_{\alpha} g_{i}\geq \textbf{1}_{K_{i}};\,\,\,\,\|g_{i}\|_{L^{p}(\mathbb{R}_{+}^{1+n})}^{p}\lesssim C_{p}^{(S_\alpha)}(K_{i}).
$$
Set
$$
g_{j,k}=\sup_{j\leq i\leq k}\left(\frac{2^{i}}{C_{S}(\mu; 2^{i})}\right)^{\frac{1}{p-q}}g_{i}
$$
for integers $j,k$ with $j<k$. Then
$$
\|g_{j,k}\|_{L^{p}(\mathbb{R}_{+}^{1+n})}^{p}\lesssim \sum_{i=j}^{k}\left(\frac{2^{i}}{C_{S}(\mu; 2^{i})}\right)^{\frac{p}{p-q}}\|g_{i}\|_{L^{p}(\mathbb{R}_{+}^{1+n})}^{p}\lesssim \sum_{i=j}^{k}\left(\frac{2^{i}}{C_{S}(\mu; 2^{i})}\right)^{\frac{p}{p-q}}C_{S}(\mu; 2^{i}).
$$
Since for $\forall \,(t,x)\in K_{i}$ and $j\leq i\leq k$ one has
$$
|S_{\alpha}g_{j,k}(t,x)|\geq \left(\frac{2^{i}}{C_{S}(\mu; 2^{i})}\right)^{\frac{1}{p-q}}S_{\alpha}g_{i}(t,x)\gtrsim \left(\frac{2^{i}}{C_{S}(\mu; 2^{i})}\right)^{\frac{1}{p-q}},
$$
it follows that
$$
2^{i}<\mu(K_{i})\leq \mu\left(E_{\left(\frac{2^{i}}{C_{S}(\mu; 2^{i})}\right)^{\frac{1}{p-q}}}(g_{j,k})\right),
$$
and so that
\begin{align*}
\|g_{j,k}\|_{L^{p}(\mathbb{R}_{+}^{1+n})}^{q}&\gtrsim \int_{\mathbb{R}_{+}^{1+n}}|S_{\alpha}g_{j,k}|^{q}d\mu\\
&\gtrsim \sum_{i=j}^{k}\left(\frac{2^{i}}{C_{S}(\mu; 2^{i})}\right)^{\frac{q}{p-q}}2^{i}\\
&\gtrsim \frac{\sum_{i=j}^{k}\left(\frac{2^{i}}{C_{S}(\mu; 2^{i})}\right)^{\frac{q}{p-q}}2^{i}\|g_{j,k}\|_{L^{p}(\mathbb{R}_{+}^{1+n})}^{q}}{\left(\sum_{i=j}^{k}\left(\frac{2^{i}}{C_{S}(\mu; 2^{i})}\right)^{\frac{q}{p-q}}C_{S}(\mu; 2^{i})\right)^{\frac{q}{p}}}\\
&\approx \left(\sum_{i=j}^{k}\frac{2^{\frac{ip}{p-q}}}{(C_{S}(\mu; 2^{i}))^{\frac{q}{p-q}}}\right)^{\frac{p-q}{p}}\|g_{j,k}\|_{L^{p}(\mathbb{R}_{+}^{1+n})}^{q}.
\end{align*}
This is the desired result thanks to
$$
\int_{0}^{\infty}\left(\frac{\lambda^{\frac{p}{q}}}{C_{S}(\mu; \lambda)}\right)^{\frac{q}{p-q}}\frac{d\lambda}{\lambda}\lesssim \sum_{i=-\infty}^{\infty}\frac{2^{\frac{ip}{p-q}}}{(C_{S}(\mu; 2^{i}))^{\frac{q}{p-q}}}\lesssim 1.
$$

Conversely, if
$$
\int_{0}^{\infty}\left(\frac{\lambda^{\frac{p}{q}}}{C_{S}(\mu; \lambda)}\right)^{\frac{q}{p-q}}\frac{d\lambda}{\lambda}<\infty,
$$
then setting
$$
T_{p,q}(\mu;g)=\sum_{i=-\infty}^{\infty}\frac{\left(\mu(E_{2^{i}}(g))-\mu(E_{2^{i+1}}(g))\right)^{\frac{p}{p-q}}}
{\left(C_{p}^{(S_\alpha)}(E_{2^{i}}(g))\right)^{\frac{q}{p-q}}}
$$
for each integer $i=0, \pm 1, \pm 2, \cdots,$ and $g\in C_{0}^{\infty}(\mathbb{R}_{+}^{1+n})$, we use an integration-by-part, the H\"{o}lder inequality and Lemma \ref{Lemma 3.2} to produce
\begin{align*}
\int_{\mathbb{R}_{+}^{1+n}}|S_{\alpha}g|^{q}d\mu&=-\int_{0}^{\infty}\lambda^{q}d\mu(E_{\lambda}(g))\\
&\lesssim \sum_{i=-\infty}^{\infty}\left(\mu(E_{2^{i}}(g))-\mu(E_{2^{i+1}}(g))\right)2^{iq}\\
&\lesssim (T_{p,q}(\mu;g))^{\frac{p-q}{p}}\left(\sum_{i=-\infty}^{\infty}2^{ip}C_{p}^{(S_\alpha)}\big(E_{2^{i}}(g)\big)\right)^{\frac{q}{p}}\\
&\lesssim (T_{p,q}(\mu;g))^{\frac{p-q}{p}}\left(\int_0^\infty C_{p}^{(S_\alpha)}(\{(t,x)\in \mathbb{R}_{+}^{1+n}:|S_{\alpha}g(t,x)|>\lambda\})d\lambda^{p}\right)^{\frac{q}{p}}\\
&\lesssim (T_{p,q}(\mu;g))^{\frac{p-q}{p}}\|g\|_{L^{p}(\mathbb{R}_{+}^{1+n})}^{q}\\
&\lesssim \|g\|_{L^{p}(\mathbb{R}_{+}^{1+n})}^{q}.
\end{align*}
In the last inequality we have used the following estimation:
\begin{align*}
(T_{p,q}(\mu;g))^{\frac{p-q}{p}}&\lesssim \left(\sum_{i=-\infty}^{\infty}\frac{\left(\mu(E_{2^{i}}(g))-\mu(E_{2^{i+1}}(g))\right)^{\frac{p}{p-q}}}
{\left(C_{S}(\mu; \mu(E_{2^{i}}(g))\right)^{\frac{q}{p-q}}}\right)^{\frac{p-q}{p}}\\
&\lesssim \left(\sum_{i=-\infty}^{\infty}\frac{(\mu(E_{2^{i}}(g)))^{\frac{p}{p-q}}-(\mu(E_{2^{i+1}}(g)))^{\frac{p}{p-q}}}
{\left(C_{S}(\mu; \mu(E_{2^{i}}(g))\right)^{\frac{q}{p-q}}}\right)^{\frac{p-q}{p}}\\
&\lesssim \left(\int_{0}^{\infty}\frac{ds^{\frac{p}{p-q}}}{(C_{S}(\mu;s))^{\frac{q}{p-q}}}\right)^{\frac{p-q}{p}}\\
&\approx \left(\int_{0}^{\infty}\left(\frac{\lambda^{\frac{q}{p}}}{C_{S}(\mu;s)}\right)^{\frac{q}{p-q}}\frac{d\lambda}{\lambda}\right)^{\frac{p-q}{p}}.
\end{align*}

Needless to say, the equivalence
$$
\eqref{eS}\Leftrightarrow P_{\alpha p}^{S}\mu \in L_\mu^{q(p-1)/(p-q)}(\mathbb{R}_{+}^{1+n})
$$
follows from Lemma \ref{l31}(b) and a modification (cf. \cite[Theorem 2.1]{COV}) of the argument for
$$
\eqref{eR}\Leftrightarrow P_{\alpha p}^{R}\mu \in L_\mu^{q(p-1)/(p-q)}(\mathbb{R}_{+}^{1+n}),
$$
and hence the interested reader can readily work out the details.
\end{proof}


\begin{thebibliography}{99}

\bibitem{Adams} D.R. Adams, Traces of potentials.II.,{\it Indiana Univ. Math. J.} 22(1973), 907--918.

\bibitem{AH}
D.R. Adams and L.I. Hedberg, {\it Function Spaces and Potential Theory}, A Series of Comprehensive Studies in Mathematics, Springer, Berlin, 1996.

\bibitem{ARAG} J.M. Angulo, M.D. Ruiz-Medina, V.V. Anh and W. Grecksch, {Fractional diffusion and fractional heat equation}, \emph{Adv. Appl. Prob.}, 32(2000), 1077--1099.

\bibitem{BC} R.M. Balan and D. Conus, A note on intermittency for the fractional heat equation, {\it Statist. Probab. Lett.}, 95(2014), 6--14.
\bibitem{BG}
R.M. Blumenthal and R.K. Getoor, Some theorems on stable processes, \emph{Trans. Amer. Math.
Soc.}, 95(1960), 263--273.

\bibitem{COV}
C. Cascante, J.M. Ortega and I.E. Verbitsky, Trace inequality of Sobolev type in the upper triangle cases, \emph{Proc. London Math. Soc.}, 80(2000), 391--414.

\bibitem{CX}
D.C. Chang and J. Xiao, $L^{q}$-extensions of $L^{p}$-spaces by fractional diffusion equations, \emph{Discrete Cont. Dyn. Systems}, 35(2015), 1905--1920.

\bibitem{CDDF}
J. Chen, Q. Deng, Y. Ding and D. Fan, Estimates on fractional power dissipatve equations
in function spaces, \emph{Nonlinear Anal.}, 75(2012), 2959--2974.


\bibitem{CS}
Z.Q. Chen and R. Song, Estimates on Green functions and Poisson kernels for symmetric
stable processes,\emph{ Math. Ann.}, 312(1998), 465--501.

\bibitem{CW}
P. Constantin and J. Wu, Behavior of solutions of 2D quasi-geostrophic equations, \emph{SIAM
J. Math. Anal.}, 148(1999), 937--948.


\bibitem{F} R. Fefferman, Strong differentiation with respect to measures, \emph{Amer. J. Math.},  103(1981), 33--40.

\bibitem{JXYZ}
R. Jiang, J. Xiao, D. Yang and Z. Zhai, Regularity and capacity for the fractional dissipative
operator, \emph{J. Differential Equ.}(2015), http://dx.doi.org/10.1016/j.jde.2015.04.033.

\bibitem{MYZ}
C. Miao, B. Yuan and B. Zhang, Well-posedness of the Cauchy problem for the fractional
power dissipative equations, \emph{Nonlinear Anal.}, 68(2008), 461--484.

\bibitem{NSS}
M. Nishio, K. Shimomura and N. Suzuki, $\alpha$-parabolic Bergman spaces, \emph{Osaka J. Math.},
42(2005), 133--162.

\bibitem{NSY}
M. Nishio, N. Suzuki and M. Yamada, Carleson inequalities on parabolic Bergman
spaces, \emph{Tohoku Math. J.}, 62(2010), 269--286.

\bibitem{NY}
M. Nishio and M. Yamada, Carleson type measures on parabolic Bergman spaces, \emph{J.
Math. Soc. Japan}, 58(2006), 83--96.

\bibitem{WY}
G. Wu and J. Yuan, Well-posedness of the Cauchy problem for the fractional power dissipative
equation in critical Besov spaces, \emph{J. Math. Anal. Appl.}, 340(2008), 1326--1335.

\bibitem{W1}
J. Wu, Lower bounds for an integral involving fractional Laplacians and the generalized
Navier-Stokes equations in Besov spaces, \emph{Comm. Math. Phys.}, 263(2005),
803--831.

\bibitem{W2}
J. Wu, Dissipative quasi-geostrophic equations with $L^p$ data, \emph{Electron. J. Diff. Equ.}, 2001(2001), 1-13.

\bibitem{XZ} L. Xie and X. Zhang, Heat kernel estimates for critical fractional diffusion operator,
arXiv:1210.7063v1.

\bibitem{Z1}
Z. Zhai, Strichartz type estimates for fractional heat equations, \emph{J. Math. Anal. Appl.}, 356(2009), 642--658.

\bibitem{Z2}
Z. Zhai, Carleson measure problems for parabolic Bergman spaces and homogeneous
Sobolev spaces, \emph{Nonlinear Anal.}, 73(2010), 2611--2630.

\end{thebibliography}
\end{document}